\documentclass[reqno,11pt]{amsart}

\usepackage[
left=1.25in, right=1.25in]{geometry}
\usepackage[all]{xy}
\usepackage{graphicx}
\usepackage{indentfirst}
\usepackage{bm}
\usepackage{amsthm}
\usepackage{mathrsfs}
\usepackage{latexsym}
\usepackage{amsmath}
\usepackage{amssymb}
\usepackage{color}
\usepackage{hyperref}
\usepackage{microtype}
\usepackage{tikz}
\usepackage{enumerate}
\usepackage{comment}

\usepackage{color}
\usetikzlibrary{calc}
\usetikzlibrary{shapes.geometric}
\usetikzlibrary{arrows}
\usetikzlibrary{decorations.pathmorphing}
\usetikzlibrary{decorations.text}
\usetikzlibrary{decorations.shapes}
\usetikzlibrary{decorations.fractals}
\usetikzlibrary{decorations.footprints}

\begin{document}

\newtheorem{theorem}{Theorem}[section]
\newtheorem{main}{Main Theorem}
\newtheorem{proposition}[theorem]{Proposition}
\newtheorem{corollary}[theorem]{Corollary}
\newtheorem{definition}[theorem]{Definition}
\newtheorem{lemma}[theorem]{Lemma}
\newtheorem{example}[theorem]{Example}
\newtheorem{remark}[theorem]{Remark}
\newtheorem{question}[theorem]{Question}
\newtheorem{conjecture}[theorem]{Conjecture}
\newtheorem{fact}[theorem]{Fact}
\newtheorem*{ac}{Acknowledgements}

\newcommand{\FS}{\mathfrak{F}_s}
\newcommand{\Z}{\mathbb{Z}}

\title{Non-commutative R\'{e}nyi Entropic Uncertainty Principles}

\author{Zhengwei Liu}
\address{Department of Mathematics and Department of Physics,
Harvard University}
\email{zhengweiliu@fas.harvard.edu}
\author{Jinsong Wu}
\address{Institute of Advanced Study in Mathematics, Harbin Institute of Technology}
\email{wjs@hit.edu.cn}

\begin{abstract}In this paper, we calculate the norm of the string Fourier transform on subfactor planar algebras and characterize the extremizers of the inequalities for parameters $0<p,q\leq \infty$.
Furthermore, we establish R\'{e}nyi entropic uncertainty principles for subfactor planar algebras.
\end{abstract}



\maketitle

\section{Introduction}
A fundamental result in quantum mechanics is Heisenberg's uncertainty principle for position and momentum.
By using the Shannon entropy of the measurement, the Hirschman-Beckner uncertainty principle was established \cite{Hirsch, Beck}.
The R\'{e}nyi entropy introduced by A. R\'{e}nyi \cite{renyi} generalized the Shannon entropy. 
In  2006, Iwo Bialynicki-Birula \cite{Birula} showed the R\'{e}nyi entropic uncertainty principles for position and momentum and also for a pair of complementary observables in $N$-level systems. 
The R\'{e}nyi entropy has been used for quantum entanglement \cite{BCHAS, GuLe}, quantum communication protocols \cite{GL, RGK}, quantum correlation \cite{LNP} , quantum measurement \cite{BG} etc.
The R\'{e}nyi entropy has applications in biology, linguistics, economics, computer sciences etc as well.
The max-entropy, the min-entropy and the collision entropy are important in quantum mechanics and they can be considered as special limits of the R\'{e}nyi entropy.

In 1936, Murray and von Neumann introduced von Neumann algebras and factors to investigate the connections between mathematics and quantum mechanics \cite{MurNeu36}. 
A subfactor is an inclusion of factors $\mathcal{N} \subset \mathcal{M}$ and its index $\delta^2$ describes the relative size of the two factors. 
Jones gave a surprising classification of the indices of subfactors in \cite{Jon83}. 
The index of a subfactor generalizes the order of a group, but it could be a non-integer which has been considered as a quantum dimension in various ways. 
Subfactor theory turns out to be a natural framework to study quantum symmetries appeared in statistical physics, conformal field theory and topological quantum field theory, see \cite{EvaKaw98}.

In \cite{JLW16}, Jiang and the authors proved various uncertain principles for subfactors in terms of Jones' planar algebras \cite{Jon99}, including the Donoho-Stark uncertainty principle for max-entropy, the Hirschman-Beckner uncertainty principle for the von Neumann entropy, and Hardy's uncertainty principle.

In the noncommutative case, the R\'{e}nyi entropy of order $p$ is defined by:
$$h_p(x)=\frac{p}{1-p}\log \|x\|_p\;, ~p\in (0,1)\cup (1,\infty),$$
where $x$ is an operator in a von Neumann algebra with a trace $tr$ and 
$$\|x\|_p=(tr_2(|x|^p))^{1/p},  ~p\in (0,\infty).$$
When $p\geq 1$, $\|x\|_p$ is called the $p$-norm of $x$.
It is natural to ask whether the R\'{e}nyi entropic uncertainty principles hold for subfactor planar algebras.
In this paper, we answer this question positively.

To establish the R\'{e}nyi entropic uncertainty principles, we calculate the norm of the string Fourier transform $\FS$ (SFT) on subfactor planar algebras.
We divide the first quadrant into three regions $R_{T},R_{F},R_{TF}$ as illustrated in Figure~\ref{Fig:SFT norm}, and let $K$ be a function on $[0,\infty)^2$ given by
\begin{equation}\label{K}
K(1/p,1/q)=\left\{
\begin{array}{lll}
\delta^{1-2/p} &\text{ for } &(1/p,1/q)\in R_F,\\
\delta^{2/q-1} &\text{ for } &(1/p,1/q)\in R_T,\\
\delta^{2/q-2/p} &\text{ for }& (1/p,1/q)\in R_{TF}.
\end{array}
\right.
\end{equation}


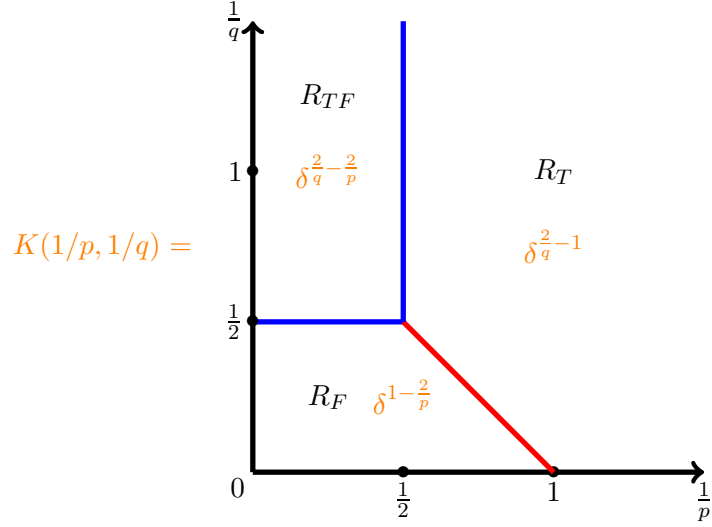
\begin{figure}
\begin{center}
\begin{tikzpicture}[scale=2]
\node at (-1,1.5) {$\textcolor{orange}{K(1/p,1/q)=}$};
\draw [line width=2] [->](0,0)--(0,3);
\node at (-0.1,-0.1) {0};
\node [left] at (0, 3) {$\frac{1}{q}$};
\draw [line width=2]  [->](0,0)--(3,0);
\node [below] at (1,0) {$\frac{1}{2}$};
\node at (1,0) {$\bullet$};
\node [below] at (2, 0) {1};
\node at (2,0) {$\bullet$};
\node [below] at (3,0) {$\frac{1}{p}$};
\draw [blue, line width=2] (1,1)--(0,1);
\node [left] at (0,1) {$\frac{1}{2}$};
\node at (0,1) {$\bullet$};
\node [left] at (0, 2) {1};
\node at (0,2) {$\bullet$};
\node [orange] at (0.5,2) {$\delta^{\frac{2}{q}-\frac{2}{p}}$};
\node at (0.5,2.5) {$R_{TF}$};
\draw [red, line width=2] (1,1)--(2,0);
\node at (0.5, 0.5) {$R_F$};
\node [orange] at (1, 0.5) {$\delta^{1-\frac{2}{p}}$};
\draw [blue, line width=2] (1,1)--(1,3);
\node [orange] at (2, 1.5) {$\delta^{\frac{2}{q}-1}$};
\node at (2,2) {$R_T$};
\end{tikzpicture}
\end{center}
\caption{The norm of the SFT.}
\label{Fig:SFT norm}
\end{figure}

\begin{theorem}[Proposition \ref{RF}, \ref{RTF}, \ref{RT}, Theorem \ref{Fourier}]
Suppose $\mathscr{P}_{\bullet}$ is an irreducible subfactor planar algebra.
Let $0<p,q\leq \infty$ and $x\in\mathscr{P}_{2,\pm}$.
Then
\begin{align}\label{Equ: HYpq}
K(1/p,1/q)^{-1}\|x\|_p\leq  \|\FS(x)\|_q\leq K(1/p,1/q)\|x\|_p.
\end{align}
\end{theorem}

\begin{theorem}[{\bf R\'{e}nyi entropic uncertainty principles}: Proposition \ref{Renyi1}]
Suppose $\mathscr{P}_{\bullet}$ is an irreducible subfactor planar algebra.
For any $x\in \mathscr{P}_{2,\pm}$ with $\|x\|_2=1$, $0< p, q<\infty$, we have that
$$(1/p-1/2)h_{p/2}(|x|^2)+(1/2-1/q)h_{q/2}\left(|\FS(x)|^2\right)\geq -\log K(1/p,1/q).$$
\end{theorem}
\noindent We also prove a second R\'{e}nyi entropic uncertainty principles:
\begin{theorem}[{\bf R\'{e}nyi entropic uncertainty principles}: Theorem \ref{renyi2}]
Suppose $\mathscr{P}_{\bullet}$ is an irreducible subfactor planar algebra.
Let $x\in \mathscr{P}_{2,\pm}$ be such that $\|x\|_2=1$.
Then for any $1/p+1/q\geq  1$, we have
$$h_{p/2}(|x|^2)+h_{q/2}(|\FS(x)|^2)\geq \left(-1+\frac{2}{2-p}+\frac{2}{2-q}\right)\log \delta^2.$$
\end{theorem}
The Donoho-Stark uncertainty principle and the Hirschman-Beckner uncertainty principle can be recovered as limits of the second R\'{e}nyi entropic uncertainty principles, see Corollary \ref{cor:DSHB}.

\begin{table}
\centering
\begin{tabular}{|c|c|}
\hline
Regions &  Extremizers (up to scale) \\
\hline
$1/p+1/q>1, 1/p>1/2$ &  trace-one projections \\
\hline
$1/p+1/q=1, 1/2<1/p<1$ & bishifts of biprojections\\
\hline
$1/p=1,1/q=0$ & extremal elements\\
\hline
$1/p=1/2, 1/q=1/2$ & $\mathscr{P}_{2,\pm}$\\
\hline
$1/p+1/q<1, 0<1/q<1/2$ & Fourier transform of trace-one projections\\
\hline
$1/q=0, 0\leq 1/p<1$ & extremal unitary elements\\
\hline
$1/q=1/2, 0\leq 1/p<1/2$ & unitary elements\\
\hline
$1/q>1/2,1/p=1/2$ & Fourier transform of unitary elements\\
\hline
$1/q>1/2,1/p<1/2$ & biunitary elements if exist\\
\hline
\end{tabular}
\caption{Characterization of the extremizers}
\label{tab1}
\end{table}

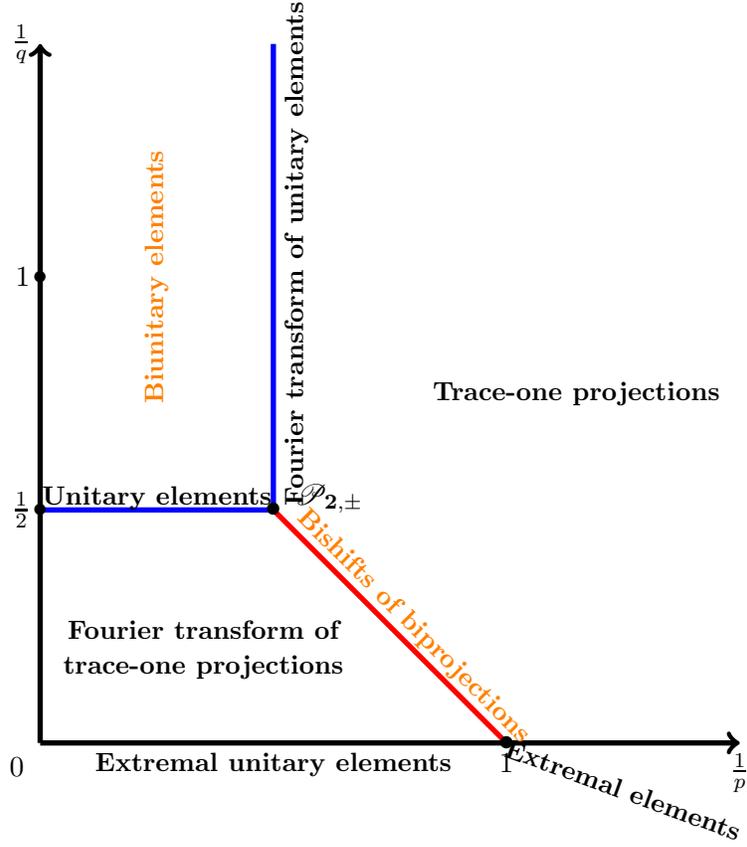
\begin{figure}

\begin{center}
\begin{tikzpicture}[scale=3.1]
\draw [line width=2] [->](0,0)--(0,3);
\node at (-0.1,-0.1) {0};
\node [left] at (0, 3) {$\frac{1}{q}$};
\draw [line width=2]  [->](0,0)--(3,0);
\node [below] at (2, 0) {1};
\node [below] at (3,0) {$\frac{1}{p}$};
\draw [blue, line width=2] (1,1)--(0,1);
\node [rotate=-45, color=orange] at (1.6,0.5) {\bf \small Bishifts of biprojections};
\node [left] at (0,1) {$\frac{1}{2}$};
\node at (0,1) {$\bullet$};
\node [left] at (0, 2) {1};
\node at (0,2) {$\bullet$};
\draw [red, line width=2] (1,1)--(2,0);
\draw [blue, line width=2] (1,1)--(1,3);
\node at (2.3,1.5) {\bf \small Trace-one projections};
\node at (1, 1) {\large $\bullet$};
\node [right] at (1.05,1.05) {$\mathbf{\mathscr{P}_{2, \pm}}$};
\node [rotate=90] at (1.1, 2.1) {\bf \small Fourier transform of unitary elements};
\node  at (0.5, 1.05) {\bf \small Unitary elements};
\node (example-align) [align=center] at (0.7, 0.4) {\bf \small Fourier transform of \\ \bf \small trace-one projections};
\node [rotate=90, color=orange] at (0.5, 2) {\bf \small Biunitary elements};
\node at (2,0) {\large $\bullet$};
\node [rotate=-20] at (2.5,-0.2) {\bf \small Extremal elements};
\node [below] at (1,0) {\bf \small Extremal unitary elements};
\end{tikzpicture}
\end{center}
\caption{Extremizers}
\label{Fig:Extremizers}
\end{figure}

We characterize the extremizers of the inequality \eqref{Equ: HYpq} for the three regions, the four critical lines and the two critical points illustrated in Figure~\ref{Fig:SFT norm}, see Table \ref{tab1} for the nine characterizations and Figure~\ref{Fig:Extremizers}. For the special case $1/p+1/q=1$, we recover the Hausdorff-Young inequality and the characterization of the extremizers as bi-shifts of biprojections in \cite{JLW16}. 

If $\mathscr{P}_{\bullet}$ is the planar algebra of a group $G$ crossed product subfactor, then the functions on $G$ are given by the 2-box space $\mathscr{P}_{2,+}$ and the representations of $G$ are characterized by the dual space $\mathscr{P}_{2,-}$. The string Fourier transform $\FS$ coincides with the classical Fourier transform.
In this way, we recover the results of Gilbert and Rzeszotnik on the norm of the Fourier transform on finite abelian groups for $1\leq p,q\leq \infty$ in \cite{GilRzes}. The notions in Table~\ref{tab1} generalize time basis, frequency basis, wave packets, biunimodular functions etc in \cite{GilRzes}. We could not find in the published literature the full results summarized in Table~\ref{tab1} for the case of finite non-abelian groups. 
For a general subfactor, both $\mathscr{P}_{2,+}$ and $\mathscr{P}_{2,-}$ could be highly non-commutative. 
All finite Kac algebras can be realized by the 2-box spaces $\mathscr{P}_{2,\pm}$ of planar algebras \cite{KLS03}.
The first R\'{e}nyi entropic uncertainty principles for locally compact quantum groups \cite{KusVae00} was studied in \cite{JLW18}, the Donoho-Stark uncertainty principles and Hirschman-Beckner uncertainty principles also are obtained. 

This paper is organized as follows. 
In Section 2, we recall some results in \cite{JLW16, JLW} on the Fourier analysis for subfactor planar algebras.
In Section 3, we calculate the norm of the Fourier transform on subfactor planar algebras and find all extremizers.
In Section 4, we prove R\'{e}nyi entropic uncertainty principles for subfactor planar algebras.

\section{Preliminaries}
We refer the reader to \cite{Jon99} for the definition of subfactor planar algebras and keep the notations in \cite{JLW}.
Suppose $\mathscr{P}_{\bullet}=\{\mathscr{P}_{n,\pm}\}_{n\geq 0}$ is a subfactor planar algebra.
Denote by $\delta$ the square root of the Jones index.
The $n$-box space $\mathscr{P}_{n,\pm}$ is a finite dimensional C$^*$-algebra.
Denote by $\mathcal{Z}(\mathscr{P}_{n,\pm})$ the center of the C$^*$-algebra $\mathscr{P}_{n,\pm}$.
Let  $tr_n$ be the (unnormalized) Markov trace on $\mathscr{P}_{n,\pm}$. Denote by $e_1$ the Jones projection in $\mathscr{P}_{2,\pm}$. 
The convolution (or coproduct) of $x, y \in \mathscr{P}_{2,\pm}$ is denoted by $x*y$.
The string Fourier transform $\FS$ (SFT) from $\mathscr{P}_{2,\pm}$ onto $\mathscr{P}_{2,\mp}$ is the clockwise 1-click rotation. The notation $SFT$ was introduced in \cite{JafLiu17} to distinguish from the quantum Fourier transform appeared in quantum information.

For any $x$ in $\mathscr{P}_{2,\pm}$, we denote by $\mathcal{R}(x)$ the range projection of $x$, $\mathcal{S}(x)$ the trace of $\mathcal{R}(x)$, $H(|x|)$ the von Neumann entropy of $|x|$, namely 
$$H(|x|)=-tr_2(|x|\log |x|).$$
A projection $B$ in $\mathscr{P}_{2,\pm}$ is a biprojection if $\FS(B)$ is a multiple of a projection.
A projection $x$ is a left shift of a biprojection $B$ if $tr_2(B)=tr_2(x)$ and $x*B=\frac{tr_2(B)}{\delta}x.$
A projection $x$ is a right shift of a biprojection $B$ if $tr_2(B)=tr_2(x)$ and $B*x=\frac{tr_2(B)}{\delta}x$.
In \cite{JLW16}, it is shown that a left shift of a biprojection is a right shift of a biprojection, where the two biprojections may be different.
A projection $x$ in $\mathscr{P}_{2,\pm}$ is a trace-one projection if $tr_2(x)=1$.
By the results in Proposition 1.9 in \cite{PmPo}, we have that a trace-one projection is a central minimal projection in $\mathscr{P}_{2,\pm}$. 
Moreover, trace-one projections are left shifts of the Jones projection $e_1$.

For a biprojection in $\mathscr{P}_{2,\pm}$, we denote by $\tilde{B}$ the range projection of $\FS(B)$.
A nonzero element $x$ in $\mathscr{P}_{2,\pm}$ is a bi-shift of a biprojection $B$ if there exists a right shift $B_g$ of the biprojection $B$ and a right shift $\tilde{B}_h$ of the biprojection $\tilde{B}$ and an element $y$ in $\mathscr{P}_{2,\pm}$ such that $x=\FS(\tilde{B}_h)*(yB_g)$.
By the results in \cite{JLW16}, there are actually eight forms of a bishift of a biprojection.
A unitary element $u\in\mathscr{P}_{2,\pm}$ is biunitary if $\FS(u)$ is a unitary. 
Biunitary elements generalize biunimodular functions for finite abelian groups.

In Fourier analysis,  the Hausdorff-Young inequality for locally compact abelian groups was studied by Hardy and Littlewood \cite{HarLittle}, Hewitt and Hirschman \cite{HeHir}, Babenko \cite{Babe}, Beckner \cite{Beck}, Russo \cite{Russo} etc.
Their results completely characterize the extremizers of the Hausdorff-Young inequality.
In \cite{JLW16}, C. Jiang and the authors prove Plancherel's formula $\|\FS(x)\|_2=\|x\|_2$ and the Hausdorff-Young inequality for subfactor planar algebras:
$$\|\FS(x)\|_q\leq \delta^{1-2/p}\|x\|_p, \quad x\in\mathscr{P}_{2,\pm},\quad 1/p+1/q=1, 1\leq p\leq 2.$$

An element $x$ in $\mathscr{P}_{2,\pm}$ is extremal if $\|\FS(x)\|_\infty=\delta^{-1}\|x\|_1$.
When $\mathscr{P}_{\bullet}$ is a group subfactor planar algebra arising from a finite abelian group, there is an explicit expression for extremal elements. (See \cite{GilRzes} for examples.)
In general, we have the following characterization:

\begin{proposition}[Corollary 6.12 and Theorem 6.13 in \cite{JLW16}]\label{maxex}
Suppose $\mathscr{P}_{\bullet}$ is an irreducible subfactor planar algebra and $w\in\mathscr{P}_{2,\pm}$.
If $\FS^{-1}(w)$ is extremal, then $wQ$ is a bishift of a biprojection, where $Q$ is the spectral projection of $|w|$  with spectrum $\|w\|_\infty$.
\end{proposition}

\begin{proposition}[Main Theorem 1,2 in \cite{JLW16}]\label{prop:DS}
Suppose $\mathscr{P}_{\bullet}$ is an irreducible subfactor planar algebra.
Then for any $x\in\mathscr{P}_{2, \pm}$, we have 
$$\mathcal{S}(\FS(x))\mathcal{S}(x)\geq \delta^2.$$
Moreover $\mathcal{S}(\FS(x))\mathcal{S}(x)=\delta^2$ if and only if $x$ is a bishift of a biprojection.
\end{proposition}

C. Jiang and the authors completely characterize the extremizers of the Hausdorff-Young inequality:

\begin{proposition}[Theorem 1.4 in \cite{JLW}]\label{Prop: hyoungeq}
Suppose $\mathscr{P}_{\bullet}$ is an irreducible subfactor planar algebra. 
Let $x$ be nonzero in $\mathscr{P}_{2,\pm}$.
Then the following are equivalent:
\begin{itemize}
\item[(1)] $\|\FS(x)\|_{\frac{p}{p-1}}=\delta^{1-2/p}\|x\|_p$ for some $1<p<2$;
\item[(2)] $\|\FS(x)\|_{\frac{p}{p-1}}=\delta^{1-2/p}\|x\|_p$ for all $1\leq p\leq 2$;
\item[(3)] $x$ is a bi-shift of a biprojection.
\end{itemize}
\end{proposition}

The 2-box space $\mathscr{P}_{2,\pm}$ is a direct sum of matrix algebras. The following proposition is a consequence of H\"{o}lder's inequality on matrix algebras.
\begin{proposition}[H\"{o}lder's inequality]\label{Prop: Holder}
Suppose $\mathscr{P}_{\bullet}$ is an irreducible subfactor planar algebra.
Then for any $x,y\in\mathscr{P}_{2,\pm}$, we have
$$\|xy\|_r\leq \|x\|_p\|y\|_q, \quad \frac{1}{r}=\frac{1}{p}+\frac{1}{q},\quad  0<r,p,q\leq \infty.$$
\begin{enumerate}[(1)]
\item If $r=1$, $1<p<\infty$, then we have $\|xy\|_1=\|x\|_p\|y\|_q$ if and only if $\frac{|x|^p}{\|x\|_p^p}=\frac{|y^*|^q}{\|y\|_q^q}$.
\item If $r=1$, $p=\infty$, then $\|xy\|_1=\|x\|_\infty\|y\|_1$ if and only if the spectral projection of $|x|$ corresponding to $\|x\|_\infty$ contains the projection $\mathcal{R}(y)$.
\end{enumerate}
\end{proposition}

\begin{definition}
Suppose $\mathscr{P}_{\bullet}$ is a subfactor planar algebra.
For any $0<p,q\leq \infty$, the norm $C_{p,q}$ of $\FS$ on $\mathscr{P}_{2,\pm}$ is defined to be
$$C_{p,q}=\sup_{\|x\|_p=1}\|\FS(x)\|_q.$$
\end{definition}
Proposition \ref{Prop: hyoungeq} shows that 
 $C_{p,q}=\delta^{1-2/p},$
when $1/p+1/q=1$ and $1\leq p\leq 2$.
We will compute $C_{p,q}$ for $0 < p,q \leq \infty $ in \S \ref{Sec: norm}.
We refer the readers to \cite{JLW16, JLW, LW17, PisXu} for other interesting inequalities on subfactor planar algebras and on non-commutative $L^p$ spaces. For example, Young's inequality has been established for subfactor planar algebras in \cite{JLW16}: 
$$\|x*y\|_r \leq \delta^{-1} \|x\|_p \|y\|_q, ~1/r+1=1/p+1/q, 1\leq p,q,r\leq \infty.$$
It would be interesting to compute
$$C_{p,q,r}:=\sup_{\|x\|_p=1, \|y\|_q=1}\|x*y\|_r,$$
for general parameters $p,q,r$.

\section{The Norm of the Fourier Transform}\label{Sec: norm}

In this section, we calculate the norm $C_{p,q}$ of the SFT. We will deal with three different cases corresponding to the three regions in Fig.~\ref{Fig:SFT norm}. Precisely
\begin{eqnarray*}
R_F:&=&\{(1/p,1/q)\in [0,\infty]^2: 1/p+1/q\leq 1,1/q\leq 1/2\},\\
R_T:&=&\{(1/p,1/q)\in [0,\infty]^2: 1/p+1/q\geq  1,1/p\geq 1/2\},\\
R_{TF}:&=&\{(1/p,1/q)\in [0,\infty]^2: 1/p\leq 1/2, 1/q\geq 1/2\}.
\end{eqnarray*}

\begin{remark}
In the finite abelian group case \cite{GilRzes}, the regions $R_F$, $R_T$, $R_{TF}$ correspond to the frequency basis, the time basis, the time-frequency basis respectively.
\end{remark}

Let $K$ be a function on $[0,\infty)^2$ defined in Equation \eqref{K}.
Then $\log K$ is an affine function in each of the three regions.

\begin{proposition}\label{RF}
Suppose $\mathscr{P}_{\bullet}$ is an irreducible subfactor planar algebra. If $(1/p,1/q)\in R_{F}$, then $C_{p,q}=\delta^{1-2/p}.$
Moreover, the following statements are equivalent:
\begin{itemize}
\item[(1)] $\|\FS(x)\|_q=\delta^{1-2/p}\|x\|_p$, for some $p> 0,q> 0$ with $1/p+1/q< 1$, $0<1/q< 1/2$;
\item[(2)] $\|\FS(x)\|_q=\delta^{1-2/p}\|x\|_p$, for all $p> 0,q> 0$ with $1/p+1/q\leq 1$, $1/q\leq 1/2$;
\item[(3)] $\FS(x)$ is a multiple of a trace-one projection.
\end{itemize}
\end{proposition}
\begin{proof}

Let $q'$ be a real number such that $1/q+1/q'=1$.
Then $1/p<1-1/q=1/q'$.
By Propositions \ref{Prop: hyoungeq}, \ref{Prop: Holder},
\begin{equation}\label{min1}
\begin{aligned}
\|\FS(x)\|_q&\leq \delta^{1-2/q'}\|x\|_{q'}\\
&\leq \delta^{1-2/q'}\|x\|_p\|1\|_{\frac{pq'}{p-q'}}\\
&=\delta^{1-2/q'}\|x\|_p\delta^{2(1/q'-1/p)}\\
&=\delta^{1-2/p}\|x\|_p.
\end{aligned}
\end{equation}

(1)$\Rightarrow$(3):
Since Inequality (\ref{min1}) becomes an equality, we have that
\begin{align}
\|\FS(x)\|_q&=\delta^{1-2/q'}\|x\|_{q'}, \label{Equ: 1} \\
\quad \|x\|_{q'}&=\|x\|_p\|1\|_{\frac{pq'}{p-q'}}. \label{Equ: 2}
\end{align}
Note that $2<q<\infty$, applying Proposition \ref{Prop: hyoungeq} to Equation (\ref{Equ: 1}), we see that $x$ is a bishift of a biprojection.
Applying Proposition \ref{Prop: Holder} to Equation (\ref{Equ: 2}), we see that $|x|$ is a multiple of $1$.
By Proposition \ref{prop:DS}, we have  $\mathcal{S}(\FS(x))=\frac{\delta^2}{\mathcal{S}(x)}=1$ and $\mathcal{R}(\FS(x))$ is a trace-one projection.
Hence $\FS(x)$ is a multiple of a trace-one projection.

(3)$\Rightarrow$(2): 
Suppose $\FS(x)$ is a trace-one projection.
Then $\|\FS(x)\|_q=1$ and $1\leq p\leq 2$, $x$ is a bishift of a biprojection as $\FS(x)$, so by Proposition 2.3 (2) $\|x\|_p=\delta^{2/p-1}$.
Hence $\|\FS(x)\|_q=\delta^{1-2/p}\|x\|_p$.
This also indicates that $C_{p,q}=\delta^{1-2/p}.$

(2)$\Rightarrow$(1): It is obvious.
\end{proof}

\begin{remark}
In the proof of Proposition \ref{RF}, when the inequality becomes equality and $q=\infty$, we have that 
$$\|\FS(x)\|_\infty=\delta^{-1}\|x\|_1, \quad \|x\|_1=\|x\|_p\|1\|_{\frac{p}{p-1}}, \quad p\neq 1.$$
We obtain that $x$ is a multiple of an extremal unitary element.
Note that if $\mathscr{P}_{\bullet}$ is a group subfactor planar algebra raising from a finite abelian group, then $x$ 
is a multiple of a character.
\end{remark}

\begin{remark} 
For the finite abelian group case, the extremizers form a basis which is a frequency basis.
But for the noncommutative case, the extremal unitaries do not form a basis in general.
\end{remark}

\begin{proposition}\label{RTF}
Suppose $\mathscr{P}_{\bullet}$ is an irreducible subfactor planar algebra and $(1/p,1/q)\in R_{TF}$. Then $C_{p,q}\leq \delta^{2/q-2/p}.$
Moreover, if there exists a biunitary in $\mathscr{P}_{2,\pm}$, the following statements are equivalent:
\begin{itemize}
\item[(1)]$\|\FS(x)\|_q=\delta^{2/q-2/p}\|x\|_p$ for some $p>0, q> 0$ with $1/p<1/2$, $1/q>1/2$;
\item[(2)]$\|\FS(x)\|_q=\delta^{2/q-2/p}\|x\|_p$ for all $p>0, q>0$ with $1/p\leq 1/2$, $1/q\geq 1/2$;
\item[(3)] $x$ is a multiple of a biunitary.
\end{itemize}
If there is a biunitary in $\mathscr{P}_{2,\pm}$, we have that $C_{p,q}=\delta^{2/q-2/p}$.
\end{proposition}
\begin{proof}

For any $1/q\geq 1/2$ and $1/p\leq 1/2$, we have
\begin{equation}\label{min2}
\begin{aligned}
\|\FS(x)\|_q&\leq \|\FS(x)\|_2\|1\|_{\frac{2q}{2-q}}\\
&=\delta^{2/q-1}\|x\|_2\\
&\leq \delta^{2/q-1}\|x\|_p\|1\|_{\frac{2p}{p-2}}\\
&=\delta^{2/q-2/p}\|x\|_p.
\end{aligned}
\end{equation}

(1)$\Rightarrow$(3):
Since Inequality (\ref{min2}) becomes equality, we have $\|\FS(x)\|_q=\|\FS(x)\|_2\|1\|_{\frac{2q}{2-q}}$ and $\|x\|_2=\|x\|_p\|1\|_{\frac{2p}{p-2}}$.
Therefore $x$ is a multiple of a biunitary.

(3)$\Rightarrow$(2): Suppose $x$ is a biunitary in $\mathscr{P}_{2,\pm}$.
Then $\|x\|_p=\delta^{2/p}$ and $\|\FS(x)\|_q=\delta^{2/q}$.
Hence $\|\FS(x)\|_q=\delta^{2/q-2/p}\|x\|_p$.
This indicates that $C_{p,q}=\delta^{2/q-2/p}$ if there exists a biunitary in $\mathscr{P}_{2,\pm}$.

(2)$\Rightarrow$(1): It is obvious.
\end{proof}

\begin{remark}
By \cite{GilRzes} if a function generates a time-frequency basis the it is unimodular (i.e. biunitary), but the converse is false. Now for the finite abelian group case, by Theorem 4.7 in \cite{GilRzes} there always exists a function which generates a time-frequency basis.
In general, there might not exist a biunitary element.
\end{remark}

\begin{remark}
Suppose $\mathscr{P}_{\bullet}=\mathscr{P}^{\mathbb{Z}_n}$ is the group subfactor planar algebra arising from the group $\mathbb{Z}_n$.
It is shown in Theorem 4.5 in \cite{GilRzes} that  $u\in\mathscr{P}_{2,\pm}$ generates a time-frequency basis if and only if
$$u(k)=\exp\left(\frac{2\pi i}{n}\left(\lambda k^2+\mu k\right)\right), \quad k\in\mathbb{Z}_n, n \text{ odd },$$
$$u(k)=\exp\left(\frac{2\pi i}{n}\left(\frac{\lambda}{2} k^2+\mu k\right)\right), \quad k\in\mathbb{Z}_n, n \text{ even },$$
where $\lambda,\mu\in\mathbb{Z}_n$ and $\lambda$ relatively prime to $n$. 
\end{remark}

\begin{remark}
Suppose $TL(\delta)$ is the Temperley-Lieb subfactor planar algebra.
Then $x\in TL(\delta)$ is a biunitary element if and only if 
$$x=1-e_1+\left(1-\frac{\delta^2}{2}\pm i\frac{\delta\sqrt{4-\delta^2}}{2}\right)e_1,\quad \delta\leq 2.$$
\end{remark}

\begin{lemma}\label{norm1}
Suppose $\mathscr{P}_{\bullet}$ is an irreducible subfactor planar algebra.
Let $x\in\mathscr{P}_{2,\pm}$.
Then for $0<p\leq 1$, we have
$$\delta^{2-2/p}\|x\|_p \leq \|x\|_1\leq \|x\|_p;$$
for $1\leq p\leq \infty$, we have
$$\|x\|_p\leq \|x\|_1\leq \delta^{2-2/p}\|x\|_p.$$
Moreover,  $\|x\|_p=\|x\|_1$, $p\neq 1$ if and only if $x$ is a multiple of a trace-one projection.
\end{lemma}
\begin{proof}
It is enough to prove for the case $x=|x|$.
Suppose that $|x|=\sum_k\lambda_k f_k\neq 0$, where $\{f_k\}_k$ is an orthogonal family of projections such that $\sum_k f_k=1$ and $\lambda_k\geq 0$.
Since $tr_2(f_k)\geq1$, we obtain the desired inequalities.
\end{proof}

\begin{proposition}\label{RT}
Suppose $\mathscr{P}_{\bullet}$ is an irreducible subfactor planar algebra and $(1/p,1/q)\in R_{T}$. 
Then $C_{p,q}= \delta^{2/q-1}.$
Moreover, the following statements are equivalent:
\begin{itemize}
\item[(1)] $\|\FS(x)\|_q=\delta^{2/q-1}\|x\|_p$, for some $p>0,q>0$ with $1/p+1/q> 1$, $1/p> 1/2$;
\item[(2)] $\|\FS(x)\|_q=\delta^{2/q-1}\|x\|_p$, for any $p> 0,q> 0$ with $1/p+1/q\geq 1$, $1/p\geq 1/2$;
\item[(3)] $x$ is a multiple of a trace-one projection.
\end{itemize}
\end{proposition}
\begin{proof}

If $1/p\leq 1$, let $p'$ be such that $1/p+1/p'=1$, then we have $1/q\geq 1-1/p=1/p'$ and
\begin{equation}\label{min3}
\begin{aligned}
\|\FS(x)\|_q&\leq \|\FS(x)\|_{p'}\|1\|_{\frac{p'q}{p'-q}}\\
&\leq\delta^{2/q-2/p'} \delta^{2/p'-1}\|x\|_p\\
&=\delta^{2/q-1}\|x\|_p.
\end{aligned}
\end{equation}
If $1/p>1$, $1/q\leq 1/2$, let $q'$ be such that $1/q+1/q'=1$, then by Lemma \ref{norm1} we have
\begin{equation}\label{min4}
\begin{aligned}
\|\FS(x)\|_q&\leq\delta^{2/q-1}\|x\|_{q'} \\
&\leq  \delta^{2/q-1}\|x\|_1\\
&\leq \delta^{2/q-1}\|x\|_p.
\end{aligned}
\end{equation}
If $1/p>1$, $1/q>1/2$, then by Lemma \ref{norm1} we have
\begin{equation}\label{min5}
\begin{aligned}
\|\FS(x)\|_q&\leq \|\FS(x)\|_2\|1\|_{\frac{2q}{2-q}}\\
&=\delta^{2/q-1}\|x\|_2\\
&\leq \delta^{2/q-1}\|x\|_1\\
&\leq \delta^{2/q-1}\|x\|_p.
\end{aligned}
\end{equation}

When Inequality (\ref{min3}) becomes equality, we have
\begin{align}\label{eq:ki1}
\|\FS(x)\|_q=\|\FS(x)\|_{p'}\|1\|_{\frac{p'q}{p'-q}},
\end{align}
\begin{align}\label{eq:ki2}
 \|\FS(x)\|_{p'}=\delta^{2/p'-1}\|x\|_p.
 \end{align}
When $p\neq 1$, by Proposition \ref{Prop: hyoungeq} and Equation \ref{eq:ki2}, we see that $x$ is a bishift of a biprojection.
By Equation \ref{eq:ki1} and Proposition \ref{Prop: Holder}, we have $|\FS(x)|$ is a multiple of $1$.
Then $x$ is a left shift of the Jones projection $e_1$ from the argument in Proposition \ref{RF}.
When $p=1$, we see that $x$ is extremal and $\FS(x)$ is a multiple of a unitary element.
By Proposition \ref{maxex}, we see that $\FS(x)$ is a bishift of a biprojection.
Hence $x$ is a multiple of a left shift of the Jones' projection, i.e. $x$ is a multiple of a trace-one projection.

When inequality (\ref{min4}) becomes equality, we have 
$$\|\FS(x)\|_q=\delta^{2/q-1}\|x\|_{q'},\quad \|x\|_1=\|x\|_p.$$
Then by Lemma \ref{norm1}, we obtain that $x$ is a multiple of a trace-one projection.

When inequality $(\ref{min5})$ becomes equality, we have
$$\|\FS(x)\|_q=\|\FS(x)\|_2\|1\|_{\frac{2q}{2-q}},\quad \|x\|_1=\|x\|_p=\|x\|_2.$$
Then by Lemma \ref{norm1}, we obtain that $x$ is a multiple of a trace-one projection.
\end{proof}

\begin{remark}
For the finite abelian group case, the extremizers form a basis which is a time basis.
But for the noncommutative case, the extremizers do not form a basis in general.
\end{remark}

\begin{proposition}
Suppose $\mathscr{P}_{\bullet}$ is an irreducible subfactor planar algebra.
Then 
\begin{itemize}
\item[(1)] $\|\FS(x)\|_2=\delta^{1-2/p}\|x\|_p$, $1/p<1/2$ if and only if $x$ is a multiple of a unitary element;
\item[(2)] $\|\FS(x)\|_q=\delta^{2/q-1}\|x\|_2$, $1/q>1/2$ if and only if $\FS(x)$ is a multiple of a unitary element.
\end{itemize}
\end{proposition}
\begin{proof}
Since
\begin{eqnarray*}
\|\FS(x)\|_2=\|x\|_2
\leq \|x\|_p\|1\|_{\frac{2p}{p-2}}
=\delta^{1-2/p}\|x\|_p,
\end{eqnarray*}
we see that $\|\FS(x)\|_2=\delta^{1-2/p}\|x\|_p$, $1/p<1/2$ if and only if $x$ is a multiple of a unitary element.

Similarly, we have that $\|\FS(x)\|_q=\delta^{2/q-1}\|x\|_2$, $1/q>1/2$ if and only if $\FS(x)$ is a multiple of a unitary element.
\end{proof}

\begin{theorem}\label{Fourier}
Suppose $\mathscr{P}_{\bullet}$ is an irreducible subfactor planar algebra.
Then for any $0<p,q\leq \infty$ we have
$$K(1/q,1/p)^{-1}\|x\|_p\leq \|\FS(x)\|_q\leq K(1/p,1/q)\|x\|_p.$$
\end{theorem}
\begin{proof}
Note that 
$$\|x\|_p=\|\FS^{-1}\FS(x)\|_p\leq K(1/q,1/p)\|\FS(x)\|_q,$$
we obtain 
$$K(1/q,1/p)^{-1}\|x\|_p\leq \|\FS(x)\|_q\leq K(1/p,1/q)\|x\|_p.$$
\end{proof}

Now the extremizers of the Fourier transform can be summarized as shown in Table \ref{tab1}.

\section{R\'{e}nyi entropy uncertainty principles}
In this section, we will show R\'{e}nyi entropic uncertainty principles for subfactor planar algebras.
First, we present the definition of R\'{e}nyi entropy for subfactor planar algebras $\mathscr{P}_{\bullet}$.
For $p\in (0,1)\cup (1,\infty)$, we define the R\'{e}nyi entropy of order $p$ of $x$ in $\mathscr{P}_{2,\pm}$ by
$$h_p(x)=\frac{p}{1-p}\log \|x\|_p.$$

\begin{proposition}[R\'{e}nyi entropic uncertainty principle]\label{Renyi1}
Suppose $\mathscr{P}_{\bullet}$ is an irreducible subfactor planar algebra.
Then for any nonzero $x\in\mathscr{P}_{2,\pm}$, we have
$$(1/p-1/2)h_{p/2}(|x|^2)+(1/2-1/q)h_{q/2}(|\FS(x)|^2)\geq  -\log K(1/p, 1/q).$$
\end{proposition}
\begin{proof}
Since for any $(1/p,1/q)\in R_F$, we have
$$\|\FS(x)\|_q\leq \delta^{1-2/p}\|x\|_p,$$
i.e.
$$\log\|\FS(x)\|_q\leq (1-2/p)\log\delta+\log \|x\|_p.$$
\begin{eqnarray*}
\lefteqn{(1/p-1/2)h_{p/2}(|x|^2)+(1/2-1/q)h_{q/2}(|\FS(x)|^2)}\\
&=&(\frac{1}{p}-\frac{1}{2})\frac{p/2}{1-p/2}\log\||x|^2\|_{p/2}+(\frac{1}{2}-\frac{1}{q})\frac{q/2}{1-q/2}\log\||\FS(x)|^2\|_{q/2}\\
&=&\log \|x\|_p-\log\|\FS(x)\|_q\\
&\geq & -(1-2/p)\log\delta.
\end{eqnarray*}
The rest of the proposition can be obtained similarly.
\end{proof}

\begin{remark}
The minimizers of the R\'{e}nyi entropic uncertainty principles in Proposition \ref{Renyi1} are the same as the extremizers for the inequalities in Theorem \ref{Fourier} given in Table \ref{tab1}.
\end{remark}

\begin{lemma}\label{limits}
Suppose $\mathscr{P}_{\bullet}$ is an irreducible subfactor planar algebra.
Let $x\in\mathscr{P}_{2,\pm}$ such that $\|x\|\leq 1$.
Then 
\begin{enumerate}
\item $h_p(x)-\frac{1}{1-p}\log\delta^2$ is a decreasing function with respect to $p$ for $p\in (0,1)\cup (1,\infty)$.
\item $\displaystyle \lim_{p\to 1}\left(h_p(x)-\frac{1}{1-p}\log \|x\|_1\right)=-\frac{tr_2(|x|\log |x|)}{\|x\|_1},$
\item $\displaystyle \lim_{p\to 0} h_p(x)=\log\mathcal{S}(x).$
\end{enumerate}
\end{lemma}
\begin{proof}
Note that 
$$\frac{d}{dp}\left(h_p(x)-\frac{1}{1-p}\log \delta^2\right)=\frac{1}{(1-p)^2}\log \frac{ tr_2(|x|^p)}{\delta^2}+\frac{1}{1-p}\frac{tr_2(|x|^p\log|x|)}{tr_2(|x|^p)}.$$
By Jensen's inequality,
\begin{eqnarray*}
\frac{d}{dp}\left(h_p(x)-\frac{1}{1-p}\log \delta^2\right)&=&  \frac{\delta^2}{p-1} \frac{ \displaystyle \frac{tr_2(|x|^p)}{\delta^2}\log\frac{ tr_2(|x|^p)}{\delta^2}-(p-1)\frac{tr_2(|x|^p\log |x|)}{\delta^2}}{(p-1)tr_2(|x|^p)}\\
&\leq&\frac{1}{p-1}\frac{tr_2(|x|^p\log |x|)}{(p-1)tr_2(|x|^p)}.
\end{eqnarray*}
When $\|x\|\leq 1$, we have that $\log |x|\leq 0$ and $\frac{d}{dp}\left(h_p(x)-\frac{1}{1-p}\log \delta^2\right)<0$.
Hence $h_p(x)-\frac{1}{1-p}\log\delta^2$ is a decreasing function.

For the first limit, we have
\begin{eqnarray*}
\lim_{p\to 1}\left(h_p(x)-\frac{1}{1-p}\log \|x\|_1\right)
&=&\lim_{p\to 1}\frac{\log tr_2(|x|^p)-\log tr_2(|x|)}{1-p}\\
&=&\left.-\frac{d}{dp}\log tr_2(|x|^p)\right|_{p=1}\\
&=&\left.-\frac{tr_2(|x|^p\log|x|)}{tr_2(|x|^p)}\right|_{p=1}\\
&=&-\frac{tr_2(|x|\log |x|)}{\|x\|_1}.
\end{eqnarray*}

For the second limit, we have that
\begin{eqnarray*}
\lim_{p\to 0}h_p(x)=\lim_{p\to 0}\frac{1}{1-p}\log tr_2(|x|^p)=\log tr_2(\mathcal{R}(|x|))=\log \mathcal{S}(x).
\end{eqnarray*}
\end{proof}

\begin{remark}
Let $x$ be nonzero in $\mathscr{P}_{2,\pm}$.
When $p=0$, the entropy 
$$h_0(x)=\lim_{p\to 0}h_p(x)=\log \mathcal{S}(x)$$ 
is called the Hartley entropy or the max-entropy of $x$.

When $p=1$,   
$$H(|x|)=\|x\|_1\lim_{p\to 1}\left(h_p(x)-\frac{1}{1-p}\log\|x\|_1\right)$$
 is the von Neumann entropy of $x$.

When $p=2$, the entropy $h_2(x)$ is called the Collision entropy of $x$.

When $p=\infty$, the entropy $\displaystyle h_\infty(x)=\lim_{p\to\infty}h_p(x)=-\log\|x\|_\infty$ is called the min-entropy of $x$.
\end{remark}

\begin{theorem}[R\'{e}nyi entropic uncertainty principle]\label{renyi2}
Suppose $\mathscr{P}_{\bullet}$ is an irreducible subfactor planar algebra.
Let $x\in \mathscr{P}_{2,\pm}$ be such that $\|x\|_2=1$.
Then for any $1/p+1/q\geq  1$, we have
$$h_{p/2}(|x|^2)+h_{q/2}(|\FS(x)|^2)\geq \left(-1+\frac{2}{2-p}+\frac{2}{2-q}\right)\log \delta^2.$$
\end{theorem}
\begin{proof}
Since $\|\FS(x)\|_{q_0}\leq \delta^{2/q_0-1}\|x\|_{q_0'}$ for any $q_0>2$, where $1/q_0+1/q_0'=1$.
Then 
$$\log \|\FS(x)\|_{q_0}\leq (2/q_0-1)\log \delta+\log \|x\|_{q_0'}$$
Hence
\begin{eqnarray*}
h_{q_0'/2}(|x|^2)+h_{q_0/2}(|\FS(x)|^2)&=&\frac{2q_0'}{2-q_0'}\log \|x\|_{q_0'}+\frac{2q_0}{2-q_0}\log\|\FS(x)\|_{q_0}\\
&=&-\frac{2q_0}{2-q_0}\log \|x\|_{q_0'}+\frac{2q_0}{2-q_0}\log \|\FS(x)\|_{q_0}\\
&\geq &-\frac{2q_0}{q_0-2}(1-2/q_0)\log \delta=2\log \delta.
\end{eqnarray*}
For each $(p,q)$ with $1/p+1/q\geq 1$, $0<p< 2$, $0<q$, we can find $(q_0',q_0)$ as above such that $1/q_0\leq 1/q$ and $1/q_0'\leq 1/p$.
Since $\|x\|_2=\|\FS(x)\|_2=1$, we have $\|x\|\leq 1$, $\|\FS(x)\|\leq 1$.
Then by decreasing in Lemma \ref{limits}, we have
\begin{eqnarray*}
h_{p/2}(|x|^2)+h_{q/2}(|\FS(x)|^2)
&\geq&h_{q_0'/2}(|x|^2)+\frac{1}{1-p/2}\log\delta^2-\frac{1}{1-q_0'/2}\log\delta^2\\
&&+h_{q_0/2}(|\FS(x)|^2)+\frac{1}{1-q/2}\log\delta^2-\frac{1}{1-q_0/2}\log\delta^2\\
&\geq &\log\delta^2+\frac{2(p-q_0')}{(2-p)(2-q_0')}\log\delta^2+\frac{2(q-q_0)}{(2-q)(2-q_0)}\log \delta^2.
\end{eqnarray*}
When $1/q<1/2$, we can take $q_0=q$, then 
\begin{eqnarray*}
\lefteqn{h_{p/2}(|x|^2)+h_{q/2}(|\FS(x)|^2)}\\
&\geq& 2\log\delta+\frac{2}{2-p}\log\delta^2-\frac{2(q-1)}{q-2}\log \delta^2\\
&=&-2\log \delta+(\frac{2}{2-p}+\frac{2}{2-q})\log \delta^2.
\end{eqnarray*}
When $1/p\leq 1$, we can take $q_0'=p$, then $1/q\geq 1-1/p=1/q_0$
\begin{eqnarray*}
\lefteqn{h_{p/2}(|x|^2)+h_{q/2}(|\FS(x)|^2)}\\
&\geq & 2\log\delta+\frac{2}{2-q}\log\delta^2-\frac{2(p-1)}{p-2}\log \delta^2\\
&=& -2\log \delta+(\frac{2}{2-p}+\frac{2}{2-q})\log \delta^2.
\end{eqnarray*}
When $1/p>1$, $1/q>1/2$, we can take 
$$q_0=\frac{2/p+2/q-3}{1/p-1},\quad q_0'=\frac{2/p+2/q-3}{1/p+2/q-2},$$
we have
\begin{eqnarray*}
\lefteqn{h_{p/2}(|x|^2)+h_{q/2}(|\FS(x)|^2)}\\
&\geq&2\log\delta+\frac{2}{2-p}\log\delta^2+\frac{2}{2-q}\log\delta^2-\frac{2/p+4/q-4}{2/q-1}\log\delta^2+2\frac{1/p-1}{-2/q+1}\log\delta^2 \\
&=&-2\log\delta+\frac{2}{2-p}\log\delta^2+\frac{2}{2-q}\log\delta^2.
\end{eqnarray*}
\end{proof}


By using the second R\'{e}nyi entropic uncertainty principles for subfactor planar algebras, we obtain the Donoho-Stark uncertainty principles and Hirschman-Beckner uncertainty principles again.
\begin{corollary}\label{cor:DSHB}
Suppose $\mathscr{P}_{\bullet}$ is an irreducible subfactor planar algebra.
Let $x\in\mathscr{P}_{2,\pm}$ be such that $\|x\|_2=1$.
Then 
$$\mathcal{S}(x)\mathcal{S}(\FS(x))\geq \delta^2,$$
and 
$$H(|x|^2)+H(|\FS(x)|^2)\geq 2\log \delta.$$
\end{corollary}
\begin{proof}
Since $\|x\|_2=1$, we have $\|x\|\leq 1$.
By Theorem \ref{renyi2}, for $p$ small enough,  we have
$$h_{p/2}(|x|^2)+h_{p/2}(|\FS(x)|^2)\geq \left(-1+\frac{2}{2-p}+\frac{2}{2-p}\right)\log \delta^2.$$
By Lemma \ref{limits}, take $p\to 0$, we have
$$\log \mathcal{S}(x)+\log\mathcal{S}(\FS(x))\geq \log\delta^2,$$
i.e. $\mathcal{S}(x)\mathcal{S}(\FS(x))\geq \delta^2$.

By Theorem \ref{renyi2}, we have 
$$h_{p/2}(|x|^2)+h_{p/2(p-1)}(|\FS(x)|^2)\geq 2\log\delta$$
By Lemma \ref{limits}, take $p\to 2$, we have that 
$$H(|x|^2)+H(|\FS(x)|^2)\geq 2\log \delta.$$
\end{proof}




\begin{ac}
Part of the work was done during visits of Zhengwei Liu and Jinsong Wu to Hebei Normal University and of Jinsong Wu to Harvard University.
Zhengwei Liu was supported by Templeton religion trust under grant TRT 0159.
Jinsong Wu was supported by NSFC (Grant no. 11771413) and partially by TRT 0159.
\end{ac}





\end{document}